\documentclass[12pt]{article}
\usepackage{amsmath, graphicx, color}
\usepackage{amstext,amsfonts,amsbsy,eucal,amssymb, amsthm}

\newtheorem{theorem}{Theorem}[section]

\theoremstyle{definition}

\newtheorem{remark}[theorem]{Remark}

\def\Z{\operatorname{\mathbb{Z}}}

\def\C{\operatorname{\mathbb{C}}}
\def\P{\operatorname{\mathbb{P}}}
\def\div{\operatorname{{\rm Div}}}
\def\pic{\operatorname{{\rm Pic}}}

\begin{document}
\begin{center}
{\large A classification of two dimensional integrable mappings and
rational elliptic surfaces}
\end{center}
\begin{center}
A. S. Carstea, T. Takenawa
\end{center}
\begin{center}
{\it National Institute of Physics and Nuclear Engineering, Dept. of Theoretical Physics, Atomistilor 407, 077125, Magurele, Bucharest, Romania}\\ 
{\it Faculty of Marine Technology, Tokyo University of Marine Science and Technology, 2-1-6 Etchu-jima, Koto-Ku, Tokyo, 135-8533, Japan}\\ 
\end{center}

\begin{abstract}
We classify two dimensional integrable mappings by investigating
the actions on the fiber space of rational elliptic surfaces. While the
QRT mappings can be restricted on each fiber, there exist several
classes of integrable mappings which exchange fibers.
We also show an equivalent condition when a generalized Halphen surface becomes a Halphen surface of index $m$.
\end{abstract}

\section{Introduction}
Discrete integrability has attracted a lot of attention in the last twenty years. Although the studies started with lattice (partial difference) equations, gradually they have been focusing on discrete non-autonomous ordinary discrete equations called mappings. Among them the most difficult and perhaps the ones which contain the very heart of integrability are discrete Painlev\'e equations. Initially, in continuous setting, Painlev\'e equations appeared as transcendental functions defined by non-autonomous nonlinear differential equations generalizing the elliptic functions (which become the solutions in the autonomous limit). These equations are integrable not in the Liouville sense (since they are non-autonomous) but in all other aspects, namely Painlev\'e property (only movable singularities are ordinary poles), monodromy preserving deformations of linear differential equations, existence of Lax pairs, $\tau$ functions, Hirota bilinear forms, B\"acklund/Schlesinger transformations etc. \cite{Ablowitz, basilbook1} for example. Moreover, they appear as similarity reductions of completely integrable partial differential equations.

In discrete setting, the first breakthrough was so called {\it singularity confinement criterion} \cite{singconf}, which imposes a {\it finite} number of iterations of singular/indeterminate behavior until it reaches regular dynamics with recovery of initial condition memory. This criterion was very productive and many of the discrete Painlev\'e equations have been obtained \cite{basilbook2-a}. Their integrable character has been settled definitively by finding Lax pairs, bilinear forms B\"acklund/Schlesinger transformations, special solutions etc.\cite{basilbook2-b}.

In \cite{Sakai01}, Sakai showed that every discrete Painlev\'e equation can be obtained as a translational component of an affine Weyl group which acts on a family of generalized Halphen surfaces, i.e. a rational surface with special divisors obtained by 9-blow-ups from $\P^2$. From this viewpoint the Quispel-Roberts-Thomson (QRT) mappings \cite{qrt} are obtained by specializations of the surfaces so that they admit elliptic fibrations.

In autonomous setting, Diller and Favre \cite{DF01} showed that if a K\"ahler surface $S$ admits an automorphism $\varphi$ of infinite order, then (i) $\varphi$ is "linearizable", i.e. it preserves the fibrations of a ruled surface \cite{TEGORS}; (ii) $\varphi$ preserves an elliptic fibration of $S$; or (iii) the algebraic (or topological) entropy of $\varphi$ is positive. The typical example of the second case is so called the QRT mappings \cite{qrt}, while mappings not belonging to the QRT family are discovered by several authors \cite{yahagy, TMNT, tsuda-takenawa, vgr,wgr}.

In this paper, we classify these types of mappings by their relation with rational elliptic surfaces. For this purpose, we consider not only rational elliptic surfaces but also generalized Halphen surfaces. In Section 2, we propose a classification of autonomous rational mappings preserving elliptic fibrations. We also show an equivalent condition when a generalized Halphen surface becomes a Halphen surface of index $m$. Although our classification is rather simple, existence of (simple) examples is nontrivial. In Section 3, we investigate properties of some known examples and construct some new examples. Section 4 is devoted to the conclusions.

Before entering into details, some preliminaries on elliptic surfaces and (generalized) Halphen surfaces may be needed.

A complex surface $X$ is called a rational elliptic surface if there exists a fibration given by the morphism: $\pi:X\rightarrow \P^1$ such that:
\begin{itemize}
\item for all but finitely many points $k\in {\P}^1$ the fibre $\pi^{-1}(k)$ is an elliptic curve;
\item $\pi$ is not birational to the projection : $E\times \P^1 \to\P^1$;
\item no fibers contains exceptional curves of first kind.
\end{itemize}
It is known that a rational elliptic surface can be obtained by 9 blow-ups from $\P^2$ and that the generic fiber of $X$ can be put into a Weierstrass form:
$$f(x,y,k)=y^2+a_1 xy+a_3 y-x^3-a_2 x^2-a_4 x-a_6,$$
where all the coefficients $a_i$ depend on $k$.
Singular fibers can be computed easily by the vanishing of the discriminant: 
$$\Delta\equiv -b_2^2 b_8-8b_4^3-27 b_6^2+9b_2b_4b_6,$$
where $b_2=a_1^2+4a_2, b_4=2a_4+a_1a_3, b_6=a_3^2+4a_6, b_8=a_1^2 a_6+4a_2a_6-a_1a_3a_4+a_2a_3^2-a_4^2.$
The discriminant has degree 12 which gives the number of singular fibers together with their multiplicities. The singularities have been classified by Kodaira according to the irreducible components of singular fibers.

A rational surface $X$ is called a generalized Halphen surface if the anti-canonical divisor $-K_X$ is decomposed into effective divisors $D_1, \dots, D_s$ as $D=\sum_{i=1}^s m_i D_i$ ($m_i\geq 1$), where the linear equivalence class of $D$ is $-K_X$ and $D_i \cdot K_X=0$ for all $i$. A generalized Halphen surface can be obtained from $\P^2$ by successive 9 blow-ups. Generalized Halphen surfaces are classified by the type of $D$ into elliptic, multiplicative and additive type (see \cite{Sakai01} for more details).

A rational surface $X$ is called a Halphen surface (or a Halphen pencil, or an elliptic surface) of index $m$ if the dimension of the linear system $|-kK_X|$ is zero for $k=1,\dots,m-1$ and $1$ for $k=m$. Here, the linear system $|-kK_X|$ is the set of curves on $\P^2$ of degree $3k$ passing through each point of blow-up (base point) with multiplicity $k$. It is known that if $m\geq 2$ a Halphen pencil of index $m$ contains a unique cubic curve $C$ with multiplicity $m$, i.e. $C$ is the unique element of $|-K_X|$. 
It is well known that if $X$ is a Halphen surface of index $m$ and $C$ is nonsingular, then $k(P_1+\cdots+P_9-3P_0)$ is not zero for $k=1,\dots,m-1$ and zero for $k=m$ (here $+$ is the group law on $C$, $P_1,\dots,P_9$ are base points of blow-ups and $3P_0$ is equal by the group law to 3 crossing points with a generic line in $\P^2$). Conversely, for a nonsingular cubic curve $C$ in $\P^2$, if $k(P_1+\cdots+P_9-3P_0)$ is not zero for $k=1,\dots,m-1$ and zero for $k=m$, then there exists a family of curves of degree $3m$ passing through $P_1,\dots,P_9$ with multiplicity $m$, which constitutes a Halphen pencil of index $m$ (see chap. 5 \S 6 of \cite{CD89} for more details).
In the next section, we extend this result into the case where $C$ is singular by using "the period map" for generalized Halphen surfaces.


\section{Classification}

Let $X$ be a rational elliptic surface obtained by 9 blow-ups from $\P^2$. 
The main result is the following classification.\\

\noindent {\bf Classification} 
Let $m$ be a positive integer, $\varphi$ an automorphism of $X$ which preserves the elliptic fibration $\alpha f_0(x,y,z)+\beta g_0(x,y,z)=0$.
Such cases are classified as follows.\\
i-$m$) 
$\varphi$ preserves $\alpha : \beta$ and the degree of fibers is $3m$;\\
ii-$m$) $\varphi$ does not preserve $\alpha: \beta$ and the degree of fibers is $3m$.\\

\begin{remark}\ 
\begin{itemize}
\item
The QRT mappings belong to Case i-1 \cite{Tsuda04}. 
\item
In case ii-$m$, elliptic fibrations admit exchange of fibers. 
\item
The integer $m$ corresponds to the index $m$ of $X$ as a Halphen surface.

\item 
As Duistermaat pointed out in \cite[\S6.3]{Duistermaat}, 
if a rational elliptic surface is a Halphen surface of index one,
or equivalently if a rational elliptic surface has a holomorphic section, 
it can be written in Wierestrass normal form.
This corresponds to Case i-1 or Case ii-1.
On the other hand, the Mordell-Weil group is a normal subgroup of 
its automorphism group that preserves each fiber,
and the Mordell-Weil groups were classified by Oguiso and Shioda \cite{og}.
Therefore Case i-1 has been completely classified in this sense. 

\item 
It is well known (for example van Hoeji's gave an algorithm \cite{Hoeij95} and Viallet et al. used it \cite{vgr}) that there exists a birational transformation 
on $\P^2$ which maps an (possibly singular) elliptic curve in $\P^2$, $\alpha f_0(x,y,z)+\beta g_0(x,y,z)=0$, into the Wierstrass normal form. 
Since in general the coefficients of this transformation are algebraic on a rational function
$\alpha/\beta=g_0(x,y,z)/f_0(x,y,z)$, 
there exists a bialgebraic transformation 
from a Halphen surface of index $m$ to that of index one that preserves the elliptic fibrations.
On the other hand, as shown in Proposition 11.9.1 of \cite{Duistermaat}, 
there does not exist a birational transformation 
from a Halphen surface of index $m$ to that of index $m'$ 
($m\neq m'$) that preserves the elliptic fibrations. 
If two infinite order mappings preserving rational elliptic fibrations are conjugate with each other by a birational mapping $\psi$, the mapping $\psi$ preserves the elliptic fibrations. Thus, two infinite order mappings belonging to different classes of the above classification are not birationally conjugate with each other.
\end{itemize}
\end{remark}

In the rest of this section, we characterize Halphen surfaces as generalized Halphen surfaces.  

Let $X$ be a generalized Halphen surface and $Q$ the root lattice defined as the orthogonal complement of $D$ with respect to the intersection form and $\omega$ a meromorphic 2-form on $X$ with $\div(\omega) = - D_{red}$, where $D_{red}=\sum_i^s D_i$. 
Then, the 2-form $\omega$ determines the period mapping $\chi$ from $Q$ to $\C$ by 
$$\chi(\alpha)=\int_\alpha \omega $$ 
in modulo $\sum_{\gamma} \Z\chi(\gamma)$, where the summation is taken for all the cycles on $D_{red}$ (see examples in the next section and \cite{Sakai01} for more details).
Note that if $X$ is not a Halphen surface of index one,
then the divisor $D$ and thus $\omega$ (modulo a nonzero constant factor) are unique.
The divisor $D$ (or $X$ itself if $X$ is not a Halphen surface of index one) is called elliptic, multiplicative, or additive type if the rank of the first homology group of $D_{red}$ is 2, 1, or 0 respectively. 

\begin{theorem}\label{thm1} \ \\
(ell)\ If a member of $|-K_X|$ is of elliptic type, then $X$ is a Halphen pencil of index $m$ iff 
$\chi(-k K_X)\neq 0$ for $k=1,\dots,m-1$ and $\chi(-m K_X)= 0$.\\
(mult)\ If a member of $|-K_X|$ is of multiplicative type, then the same assertion holds as in the elliptic case.\\
(add)\ If a member of $|-K_X|$ is of additive type, then $X$ is a Halphen pencil of index 1 iff $\chi(-K_X)=0$, and never a Halphen pencil of index $m\geq 2$.
\end{theorem}

\begin{proof}
Case (ell) is a classical result (see Remark 5.6.1 in \cite{CD89} or references therein).
Case (mult) and case (add) of index 1 are Proposition 23 in \cite{Sakai01}.   
Similar to that proof, we can vary $D$ and $\chi$ continuously to nonsingular case.
Indeed, let $P_1,\dots,P_9$ be the points of blow-ups (possibly infinitely near, we assume $P_9$ is the point for the last blow-up) and $f_0$ be the cubic polynomial defining $D$.
There exists a pencil of cubic curves $C_\lambda: f_\lambda=f_0+\lambda f_1=0$ $\lambda \in \P^1$ passing through the 8 points $P_1,\dots,P_8$. For small $\lambda$, the cubic curve $C_\lambda$ is close to $D$, and the meromorphic 2-form $\omega_\lambda$ for $C_\lambda$ is also close to $\omega$. Let $P_9'$ be a point close to $P_9$ on $C_\lambda$ such that
$$\lim_{\lambda\to 0}\chi_\lambda(-mK_{X'})= \lim_{\lambda\to 0} \int_{-mK_{X'}} \omega' = \int_{-mK_X} \omega' = \chi(-mK_X) 
$$ 
holds,y
where $X'$ is the surface obtained by blow-ups at $P_1,\dots,P_8$ and $P_9'$ instead of $P_9$. 
Thus, $\chi_\lambda (-mK_{X'})\neq 0$ holds if $\chi(-mK_{X})\neq 0$ for small $\lambda$, and therefore $X$ does not have a pencil of degree $3m$.
Conversely, if $\chi(-mK_{X})= 0$, then $\chi'(-mK_{X'})$ is close to zero, and there exists $P_9''$ close to $P_9'$ on $C'$ such that $\chi_\lambda(-mK_{X''})= 0$. Thus, we have 
$$\lim_{\lambda\to 0}\chi_\lambda(-mK_{X''})= \chi(-mK_X).$$  
Since $X''$ has (at least) a pencil of curves of degree $3m$ passing through the 9 points with multiplicity $m$ and this condition is closed in the space of coefficients of polynomials defining curves, $X$ also has the same property.
\end{proof}

\begin{remark}\label{q}
In Painlev\'e context, for multiplicative case, $\chi$ is normalized so that $\chi(\gamma)=2\pi i$ for a simply connected cycle $\gamma$ on some $D_i$, and the parameter ``$q$'' is defined as $q=\exp \chi(-K_X)$, i.e. the condition $\chi(-mK_X)=0$ corresponds to $q^m=1$.  
We must point out here
that in \cite{grt} and \cite{grtw} similar study has been done on $q$-Painlev\'e equations, and it is reported that Eq. (3.1) of \cite{grt} with $q=\sqrt{-1}$ preserves degree (4,4) pencil, which seems contradict to the above theorem, but there the definition of $q$ is different from ours (its square root is our $q$).
\end{remark}

\section{Examples}

In this section, we are going to give examples for case i-2, ii-1 and ii-2.
A typical example of Case i-1 is the QRT mappings. There is some literature on their relation to rational elliptic surfaces \cite{Tsuda04, Duistermaat}, and we are not going to discuss it here.
In the first subsection, we investigate the action on the space of initial conditions of some mapping of Case ii-1, which was proposed in \cite{tsuda-takenawa}. In the second subsection, we show that one of the HKY mappings belongs to Case i-2. Theoretically, from Theorem~\ref{thm1}, we can construct mappings of the type i-$m$ for any integers. Actually, let $\phi(q)$ be some $q$-discrete Painlev\'e equation and $q$ a primitive $m$-th root of unity, then
$\phi^m(q)$ is autonomous and preserves the Halphen fibration of index $m$.
However, the degrees of mappings obtained in this way are very high.
The HKY mapping is much simpler example. In the third subsection, we construct some example for Case ii-2, which we believe as the first example for this case.

\subsection{Case ii-1}

We start with a mapping \cite{stef, tsuda-takenawa,grt} which preserves elliptic fibration of degree $(2,2)$ but exchanges the fibers:
\begin{align}\label{ii-1}
x_{n+1}&=-x_{n-1}\frac{(x_n-a)(x_n-1/a)}{(x_n+a)(x_n+1/a)}.
\end{align}
In this subsection, studying space of initial conditions (values), we compute the conserved quantity, the parameter ``$q$'' and all singular fibers. We also clarify the relation with the $q$-discrete Painlev\'e VI equation ($qP(A_3^{(1)})$) by Sakai's notation by deautonomizing the mapping \eqref{ii-1}, where the label $A_3^{(1)}$ corresponds to the type of space of initial conditions.

First of all, in order to compactify the space of dependent variables, we write the equations in projective space as a two component system:
$$\phi:\P^1\times \P^1\to \P^1\times \P^1, \phi(x,y)=(\overline x,\overline y),$$
\begin{align}
\overline x&=y \nonumber\\
\overline y&=-x \frac{(y-a)(y-1/a)}{(y+a)(y+1/a)}.
\end{align}
We use $\P^1\times \P^1$ instead of $\P^2$ just because the parameters of blowing-up points become easy to write.
The projective space $\P^1\times \P^1$ is generated by the following coordinate system ($X=1/x,Y=1/y$):
$$\P^1\times \P^1=(x,y)\cup(X,y)\cup(x,Y)\cup(X,Y).$$
The indeterminate points for the mappings $\phi$ and $\phi^{-1}$ are
\begin{eqnarray*}
P_1:(x,y)=(0,-a),&P_2:(x,y)=(0,-1/a),\\
P_3:(X,y)=(0,a),&P_4:(X,y)=(0,1/a),\\
P_5:(x,y)=(a,0),&P_6:(x,y)=(1/a,0),\\
P_7:(x,Y)=(-a,0),&P_8:(x,Y)=(-1/a,0).
\end{eqnarray*}
Let $X$ be the surface obtained by blowing up these points. Then,  
$\phi$ is lifted to an automorphism of $X$. Such a surface $X$ is called the space of initial conditions. More generally, if a sequence of mappings $\{\phi_n\}$ is lifted to a sequence of isomorphisms from a surface $X_n$ to a surface $X_{n+1}$, each surface $X_n$ is called the space of initial conditions. 

The Picard group of $X$ is a $\Z$-module:
$$\pic(X)=\Z H_x\oplus \Z H_y\oplus \bigoplus_{i=1}^{8}\Z E_i,$$
where $H_x$, $H_y$ are the total transforms of the lines $x={\rm const.}$, $y={\rm const.}$ and $E_i$ are the total transforms of the eight points of blow-ups. 
The intersection form of divisors is given by $H_z\cdot H_w=1-\delta_{zw},\quad E_i\cdot E_j=-\delta_{ij},\quad H_z\cdot E_k=0$ for $z,w=x,y$.
Also the anti-canonical divisor of X is 
$$-K_X=2H_x+2H_y-\sum_{i=1}^{8}E_i.$$

Let us denote an element of the Picard lattices by $A=h_0 H_x+h_1 H_y+\sum_{i=1}^8 e_i E_i$ ($h_i,e_j\in{\bf Z}$), then the induced bundle mapping is acting on it as 
\begin{align*}
&\phi_{*}(h_0,h_1,e_1,...,e_8)\\
=&(h_0,h_1,e_1,...,e_8)
\left(
\begin{array}{cccccccccc}
2&1&0&0&0&0&-1&-1&-1&-1\\
1&0&0&0&0&0&0&0&0&0\\
1&0&0&0&0&0&0&0&-1&0\\
1&0&0&0&0&0&0&0&0&-1\\
1&0&0&0&0&0&-1&0&0&0\\
1&0&0&0&0&0&0&-1&0&0\\
0&0&1&0&0&0&0&0&0&0\\
0&0&0&1&0&0&0&0&0&0\\
0&0&0&0&1&0&0&0&0&0\\
0&0&0&0&0&1&0&0&0&0\\
\end{array}\right).
\end{align*}
It preserves the decomposition of $-K_X=\sum_{i=0}^3 D_i$:
\begin{align}
&D_0=H_x-E_1-E_2,\ D_1=H_y-E_5-E_6 \nonumber \\
&D_2=H_x-E_3-E_4,\ D_3=H_y-E_7-E_8, \label{Dii-1}
\end{align}
which constitute the $A_3^{(1)}$ type singular fiber: $xy=1$.

One can see that the elliptic curves
\begin{align*}
F\equiv &\alpha xy-\beta((x^2+1)(y^2+1)+(a+1/a)(y-x)(xy+1))=0\\
&\Leftrightarrow \  k xy-((x^2+1)(y^2+1)+(a+1/a)(y-x)(xy+1))=0
\end{align*}
correspond to the anti-canonical class (these curves pass through all $E_i$'s for any $(\alpha:\beta$). 
This family of curves defines a rational elliptic surface. One can see that even though the anti-canonical class is preserved by the mapping, the each fiber is not. More precisely, the action changes $k$ in $-k$.

So, as a conclusion, the dimension of the linear system corresponding to the anti-canonical divisor is $1$. It can be written as
$\alpha f_1(x,y)+\beta f_2(x,y)=0\Leftrightarrow k f_1(x,y)+f_2(x,y)=0$ for $\alpha:\beta \in \P(\C)$ and 
$\deg f=\deg g =(2,2)$. This elliptic fibration is preserved by the action of the dynamical system but not trivially in the sense that the fibers are exchanged. The conserved quantity becomes higher degree as $(f/g)^\nu$ for some $\nu>1$. In our case $\nu=2$ and the invariant is exactly the same as the result of \cite{tsuda-takenawa}.

\begin{remark}
In order to have a Weierstrass model, we perform some homographic transformations according to the algorithm of Schwartz \cite{schwarz}. Then, after long but straightforward calculations, we can compute the roots of the elliptic discriminant $\Delta (k)$ as
\begin{align*}
k_1=0, \quad{\rm multiplicity}=2&\\
k_{2,3}=\pm 4(1+a^2)/a, \quad{\rm multiplicity}=1&\\
k_{4,5}=\pm (1-a^2)^2/a^2,\quad {\rm multiplicity}=2&\\
k_6=\infty,\quad {\rm multiplicity}=4&.
\end{align*}
We have $A_1^{(1)}$ singular fiber for $k_1$ and $k_{4,5}$, $A_0^{(1)}$ fiber for $k_{2,3}$, and $A_3^{(1)}$ fiber for $k_6$. The mapping acts on these singular fibers as an exchange as ($k_1\to k_1, k_2\rightarrow k_3\rightarrow k_2, k_4\rightarrow k_5\rightarrow k_4, k_6\to k_6$).
\end{remark}
 
\begin{remark}
If a surface is a generalized Halphen surface but not a Halphen surface of index $1$, then the anti-canonical divisor $-K_X$ is uniquely decomposed to a sum of effective divisors as $-K_X=\sum m_i D_i$ and we can characterize the surface by the type the decomposition. However, if the surface is Halphen of index $1$, it may have several types of singular fibers as this example. 
\end{remark}

Next, we consider deautonomization of the mapping $\phi$. 
For that, we will use the decomposition \eqref{Dii-1} of $-K_X$ which is preserved by the mapping, though the decomposition and hence the deautonomization are not unique (the fiber corresponding to $k=0$ is also preserved).   

The affine Weyl group symmetries are related to the orthogonal complement of $D_{red}=\{D_1,\dots,D_4\}$. In order to see this, we note that ${\rm rank}\ \pic(X)={\rm rank}\ \langle H_x,H_y,E_1,...E_8\rangle_{\Z}=10$.
The orthogonal complement of $D_{red}$:
\begin{align*}
\langle D\rangle ^{\bot}&=\{\alpha\in \pic(X)|\alpha\cdot D_i=0, i=0,3\}
\end{align*}
has 6-generators:
\begin{align*}
&\langle D\rangle ^{\bot}=\langle \alpha_0,\alpha_1,...,\alpha_5\rangle _{\Z}\\
&\alpha_0=E_1-E_2,\alpha_1=E_3-E_4,\alpha_2=H_y-E_1-E_3\\
&\alpha_3=H_x-E_5-E_7,\alpha_4=E_5-E_6,\alpha_5=E_7-E_8.
\end{align*}
\begin{figure}[ht]
\begin{center}
\includegraphics[width=6cm]{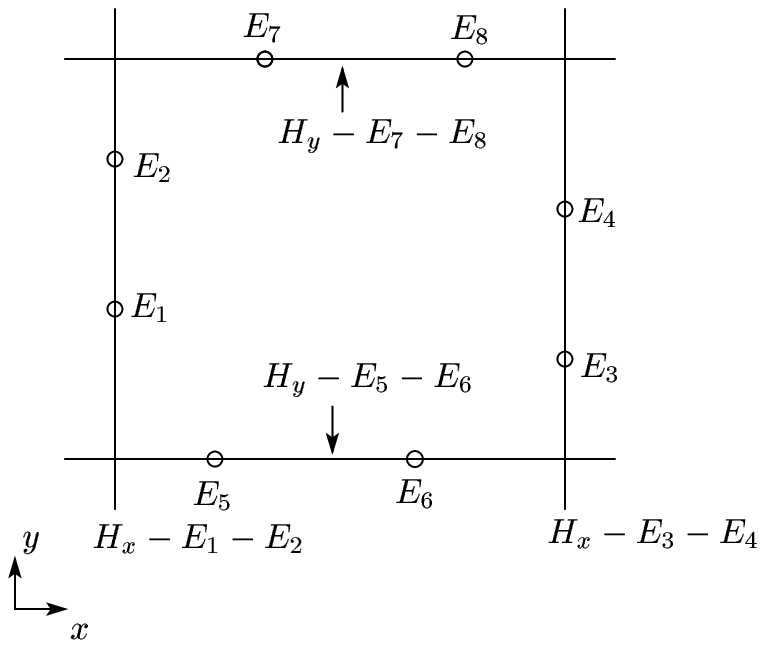}
\mbox{\raisebox{10mm}{\includegraphics[width=4cm]{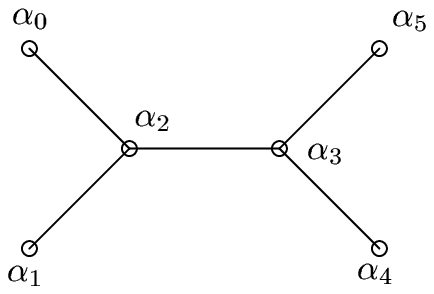}}}
\caption{Singular fiber and orthogonal complement.}
\end{center}
\end{figure}

Related to them, we define elementary reflections:
$$w_i:\pic(x)\to \pic(X), w_i(\alpha_j)=\alpha_j-c_{ij}\alpha_i,$$
where $c_{ji}=2(\alpha_j\cdot\alpha_i)/(\alpha_i\cdot\alpha_i)$. One can easily see that $c_{ij}$ is a Cartan matrix of $D_5^{(1)}$-type for the root lattice
$Q=\bigoplus_{i=0}^5 \Z\alpha_i$. We also introduce permutations of roots:
\begin{align*}
&\sigma_{10}:(\alpha_0,\alpha_1,\alpha_2,\alpha_3,\alpha_4,\alpha_5)\mapsto(\alpha_1,\alpha_0,\alpha_2,\alpha_3,\alpha_4,\alpha_5)\\
&\sigma_{tot}:(\alpha_0,\alpha_1,\alpha_2,\alpha_3,\alpha_4,\alpha_5)\mapsto(\alpha_5,\alpha_4,\alpha_3,\alpha_2,\alpha_1,\alpha_0).
\end{align*}
The group generated by reflections and permutations becomes an extended affine Weyl group:
$$\widetilde W(D_5^{(1)})=\langle w_0,w_1,...,w_5,\sigma_{10},\sigma_{tot}\rangle.$$

This extended affine Weyl group can be realized as an automorphisms of a family of generalized Halphen surfaces which are obtained by allowing the points of blow-ups to move so that they preserve the decomposition of $-K_X$ as
\begin{eqnarray*}
P_1:(x,y)=(0,a_1),&P_2:(x,y)=(0,a_2),\\
P_3:(X,y)=(0,a_3),&P_4:(X,y)=(0,a_4),\\
P_5:(x,y)=(a_5,0),&P_6:(x,y)=(a_6,0),\\
P_7:(x,Y)=(a_7,0),&P_8:(x,Y)=(a_8,0),
\end{eqnarray*}
which can be normalized as $a_1a_2a_3a_4=a_5a_6a_7a_8=1$.
Accordingly, our mapping {\it lives} in an extended affine Weyl group $\widetilde W(D_5^{(1)})$ and deautonomized as
\begin{align*}
\tilde{\phi}:\ \left\{
\begin{array}{rl}
\overline x&=a_1a_2y \nonumber\\
\overline y&=\displaystyle -x \frac{(y-a_3)(y-a_4)}{(y-a_1)(y-a_2)}
\end{array}\right.
\end{align*}
with
\begin{align*}
&(a_1,a_2,a_3,a_4,a_5,a_6,a_7,a_8,q)\\
&\mapsto
(-\frac{1}{\sqrt{q}a_6},-\frac{1}{\sqrt{q}a_5},-\frac{\sqrt{q}}{a_8},-\frac{\sqrt{q}}{a_7},a_3,a_4,a_1,a_2,q),
\end{align*}
where
\begin{align}\label{qii-1}
q=\frac{a_1a_2a_7a_8}{a_3a_4a_5a_6}.
\end{align}
This mapping can be decomposed by elementary reflections as
\begin{align*}
\tilde{\phi}_{*}=&\sigma_{10}\circ \sigma_{tot}\circ \sigma_{10}\circ \sigma_{tot}\circ w_2\circ w_1\circ w_0\circ w_2\circ w_1\circ w_0
\end{align*}
and acts on the root lattice as 
\begin{align*}
&(\alpha_0,\alpha_1,\alpha_2,\alpha_3,\alpha_4,\alpha_5)\\
&\mapsto (-\alpha_5,-\alpha_4,-\alpha_3,\alpha_2+2\alpha_3+\alpha_4+\alpha_5,\alpha_0,\alpha_1).
\end{align*}
Hence, $\tilde{\phi}^4$ is a translational element of the extended affine Weyl group, and therefore one of the $q$-Painlev\'e VI equations ($qP(A_3^{(1)})$) in
Sakai's sense, while the original $q$-Painlev\'e VI studied 
in \cite{Sakai01} was
\begin{align*}
qP_{{\rm VI}}:\ \left\{
\begin{array}{rl}
\overline x&=\displaystyle -\frac{y}{a_1a_2} \nonumber\\
\overline y&=\displaystyle  -\frac{(y-a_1)(y-a_2)}{x(y-a_3)(y-a_4)}
\end{array}\right.
\end{align*}
with
\begin{align*}
&(a_1,a_2,a_3,a_4,a_5,a_6,a_7,a_8,q)\\
&\mapsto
(-\sqrt{q}a_5,-\sqrt{q}a_6,-\frac{a_7}{\sqrt{q}},-\frac{a_8}{\sqrt{q}},-a_1,-a_2,-a_3,-a_4,q),
\end{align*}
which is decomposed by elementary reflections as
\begin{align*}
qP_{{\rm VI}}&=\sigma_{10}\circ w_1\circ  w_0\circ  w_2\circ w_1\circ w_0\circ w_2\circ w_1\circ w_0
\end{align*}
and acts on the root lattice as
\begin{align*}
&(\alpha_0,\alpha_1,\alpha_2,\alpha_3,\alpha_4,\alpha_5)\\
&\mapsto (-\alpha_4,-\alpha_5,-\alpha_3,-\alpha_2+\delta,-\alpha_0,-\alpha_1)\\
&(\delta=\alpha_0+\alpha_1+2\alpha_2+2\alpha_3+\alpha_4+\alpha_5).
\end{align*}

At the last of this subsection, we define the period map $\chi:Q\to \C$ and
compute $q$ \eqref{qii-1} by using
\begin{align}
\omega&=\frac{1}{2\pi i}\frac{dx\wedge dy}{xy}.\label{omega}
\end{align}
For example, $\chi(\alpha_0)$ is computed as follows.
The exceptional divisors $E_1$ and $E_2$ intersect with $D_0$
at $(x,y)=(0,a_1)$ and $(0,a_2)$, and  $\chi(\alpha_0)$ is computed as
\begin{align*}
\chi(\alpha_0)=&\int_{|x|=\varepsilon,\ y=a_2\sim a_1} \frac{1}{2\pi i}\frac{dx\wedge dy}{xy}\\   
=& - \int_{a_2}^{a_1} \frac{dy}{y} \\
=&\log \frac{a_2}{a_1},   
\end{align*}
where $y=a_2\sim a_1$ denotes a path from $y=a_2$ to $y=a_1$ in $D_1$. According to the ambiguity of paths, the result should be considered in modulo $2\pi i \Z$. Similarly, we obtain
\begin{align*}
& \chi(\alpha_0)=\log \frac{a_2}{a_1},\
\chi(\alpha_1)=\log \frac{a_3}{a_4},\
\chi(\alpha_2)=\log \frac{a_1}{a_3},\\
&\chi(\alpha_3)=\log \frac{a_7}{a_5},\
\chi(\alpha_4)=\log \frac{a_5}{a_6},\
\chi(\alpha_5)=\log \frac{a_8}{a_7},
\end{align*}
and therefore we have 
$$\chi(-K_X)=\log \frac{a_1a_2a_7a_8}{a_3a_4a_5a_6}
$$ 
and $q$ as \eqref{qii-1} (see Remark~\ref{q}).
For the mapping $\phi$, we have $q=1$.

\subsection{Case i-2}
We consider the following HKY mapping which is a symmetric reduction of $qP_{{\rm V}}$ for $q=-1$ \cite{yahagy} (also \cite{basilbook2-a} pg. 311).
\begin{align}
\bar{x}&=\frac{(x-t)(x+t)}{y(x-1)}\nonumber\\
\bar{y}&=x \label{i-2}
\end{align}
We define the space of initial conditions as a rational surface obtained by blow-ups from 
$\P^1\times \P^1$ at 8 points:
\begin{eqnarray*}
&P_1:(x,y)=(a_1,0)=(t,0),
&P_2:(x,y)=(a_2,0)=(-t,0)\\
&P_3:(x,y)=(0,a_3)=(0,t),
&P_4:(x,y)=(0,a_4)=(0,-t)\\
&P_5:(x,y)=(1,\infty),
&P_6:(x,y)=(\infty,1)\\
&P_7:(x,y)=(\infty,\infty),
&P_8:(x,x/y)=(\infty,a_5)=(\infty,1),
\end{eqnarray*}
where $a_i$'s wii be used for deautonomization later.
The system acts the surface as a holomorphic automorphism.

Again we investigate the linear system of the anti-canonical divisor class
$-K_X=2H_x+2H_y-E_1-\cdots-E_8$.
For the example, $\dim -K_X$ is zero and $\dim -2K_X$ is one. 
Actually, we have 
\begin{align*}
|-2K_X|=&\alpha x^2y^2 +\beta (2x^2y^3+2x^3y^2+x^2y^4+x^4y^2-2x^3y^3-\\&
2xy^4-2x^4y+x^4+y^4+2t^2(x y^2+x^2y-y^2-x^2)+t^4)\equiv \alpha f+\beta g,
\end{align*}
and  
\begin{align*}k=\frac{g}{f}=&\Big( 2x^2y^3+2x^3y^2+x^2y^4+x^4y^2-2x^3y^3-
2xy^4-2x^4y\\
&+x^4+y^4+2t^2(x y^2+x^2y-y^2-x^2)+t^4)\Big)\Big/ (x^2y^2)
\end{align*}
is the conserved quantity. So it belongs to Case ii-1.

\begin{remark}\ 
We say a curve $f(x,y)=0$ passes through a point $(x_0,y_0)$ 
with multiplicity $m$ if 
$\frac{\partial^j f(x_0,y_0)}{\partial x^p \partial y^q}=0$ for any
$j\leq m$ and $p+q=j$. 
The calculation of multiplicity at $P_8$ is very sensitive.
For example, for $f(x,y)=x^2y^2$, we relate $F(X,Y)=X^2Y^2$ so that the sum of degrees is $(4,4)$. Since this curve passes through $P_7$ with multiplicity 2, the proper transform of the curve in the coordinate: 
$(u,v)=(X,Y/X)$ is given by $F(u,uv)/u^2=u^2v^2$, which passes through
$P_8:(u,v)=(0,1)$ with multiplicity 2.
\end{remark}

\begin{remark}
The unique anti-canonical divisor $-K_X$ is decomposed as 
$-K_X=\sum_{i=0}^4 D_i$ by
\begin{align*}
&D_0=H_y-E_1-E_2,\ D_1=H_x-E_6-E_7\\
&D_2=E_7-E_8,\ D_3=H_y-E_5-E_6,\ D_4=H_x-E_3-E_4,
\end{align*}
which constitute $A_4^{(1)}$-type singular fiber $xy=1$.
The orthogonal complement of $D_i$'s is generated by
\begin{align*}
&\alpha_0=H_x+H_y-E_1-E_3-E_7-E_8,\ \alpha_1=E_1-E_2\\
&\alpha_2=H_x-E_1-E_5, \alpha_3=H_y-E_3-E_6,\ \alpha_4=E_3-E_4,
\end{align*}
which forms the Dynkin diagram of type $D_4^{(1)}$.
Let $\omega$ be the same as \eqref{omega}, then 
the period map $\chi:Q\to \C$ for $Q=\bigoplus_{i=0}^4 \Z\alpha_i$ 
is computed  as
\begin{align*}
& \chi(\alpha_0)=\log \left(-\frac{a_3}{a_1a_5}\right),\
\chi(\alpha_1)=\log \frac{a_1}{a_2},\
\chi(\alpha_2)=\log \frac{1}{a_1},\\
&\chi(\alpha_3)=\log a_3,\
\chi(\alpha_4)=\log \frac{a_4}{a_3},
\end{align*}
and therefore $$\chi(-K_X)=\log \frac{a_3a_4}{a_1a_2a_5}+\pi i$$
Hence we can take $q$ as 
$$q=-\frac{a_3a_4}{a_1a_2a_5}$$ 
and we have $q=-1$ for the original mapping.

\begin{figure}[ht]
\begin{center}
\includegraphics[width=6cm]{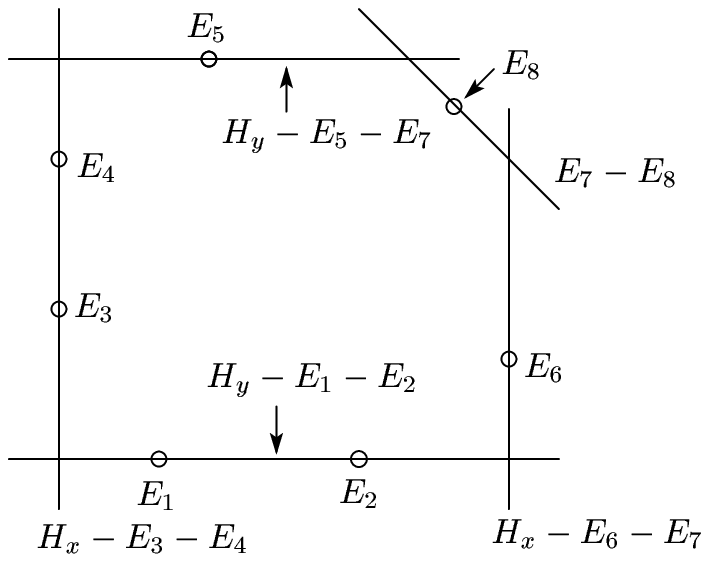}
\mbox{\raisebox{10mm}{\includegraphics[width=4cm]{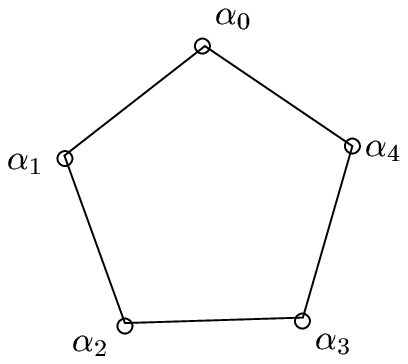}}}
\caption{Singular fiber and orthogonal compliment.}
\end{center}
\end{figure}
\end{remark}

The mapping \eqref{i-2} can be deautonomized to one of $qP_{{\rm V}}$ ($qP(A_4^{(1)}$)) as
\begin{align*}
\begin{array}{rl}
\overline x&=\displaystyle \frac{a_5(x-a_1)(x-a_2)}{y(x-1)} \nonumber\\
\overline y&=x
\end{array}
\end{align*}
with
\begin{align*}
&(a_1,a_2,a_3,a_4,a_5,q)
\mapsto
(\frac{a_4}{q},\frac{a_3}{q},a_1,a_2,-\frac{1}{a_5},q),
\end{align*}
which acts on the root lattice as
\begin{align*}
&(\alpha_0,\alpha_1,\alpha_2,\alpha_3,\alpha_4)
\mapsto (\alpha_2+\alpha_3+\alpha_4,-\alpha_4,-\alpha_3,\alpha_0+\alpha_3+\alpha_4,\alpha_1).
\end{align*}
While the original $qP_{{\rm V}}$ mapping was
\begin{align*}
\begin{array}{rl}
\overline x&=\displaystyle \frac{a_5(x-a_1)(x-a_2)}{y(x-1)} \nonumber\\
\overline y&=x
\end{array}
\end{align*}
(same with the above) with
\begin{align*}
&(a_1,a_2,a_3,a_4,a_5,q)
\mapsto
(\frac{a_4}{q},\frac{a_3}{q},a_2,a_1,-\frac{1}{a_5},q),
\end{align*}
which acts on the root lattice as
\begin{align*}
&(\alpha_0,\alpha_1,\alpha_2,\alpha_3,\alpha_4)\\
&\mapsto (\alpha_1+\alpha_2+\alpha_3+\alpha_4,-\alpha_4,-\alpha_3,\alpha_0+\alpha_1+\alpha_3+\alpha_4,-\alpha_1).
\end{align*}

\subsection{Case ii-2}
We consider the following mapping $\varphi$: 
\begin{align}
\varphi:&\left\{ \begin{array}{ll} \vspace{2mm}
\bar{x}&= \displaystyle
\frac{x(-ix(x+1) + y(bx+1))}{y( x(x-b)+ iby(x-1))}\\ 
\bar{y}&= \displaystyle
\frac{x(x(x+1) + iby(x-1))}{b(x(x+1) - iy(x-1))}\end{array}\right.,
\label{exii-2}
\end{align}
which is obtained by specializing
one of $qP(A_5^{(1)})$ equation.
Notice that the space of initial conditions for 
both $qP_{{\rm III}}$ and $qP_{{\rm IV}}$ is
the generalized Haphen surface of type of $A_3^{(1)}$ \cite{Sakai01}, and thus we may not be able to say that 
a translational element of the corresponding affine Weyl group
is one of $qP_{{\rm III}}$ equations or $qP_{{\rm IV}}$ equations.

The inverse of $\varphi$ is
\begin{align}
\varphi^{-1}:&\left\{ \begin{array}{ll} \vspace{2mm}
\underline{x}&= \displaystyle
\frac{y(bxy-bx - by+1)}{xy -x+ by -1}\\ 
\underline{y}&= \displaystyle
\frac{-iy(bxy-bx - by+1)(bxy+x-by+1)}{
bx(xy-x-y -1)(xy-x+by-1)}\end{array}\right.
\end{align}
and the space of initial conditions is obtained by blow-ups from 
$\P^1\times \P^1$ at 8 points:
\begin{eqnarray*}
&P_1:(x,y)=(-1,0),
&P_2:(x,y)=(0,1/b)\\
&P_3:(x,y)=(1,\infty),
&P_4:(x,y)=(\infty,1)\\
&P_5:(x,y)=(0,0),
&P_6:(x,y/x)=(0,i)\\
&P_7:(x,y)=(\infty,\infty),
&P_8:(x,x/y)=(\infty, -ib).
\end{eqnarray*}
Then $\varphi$ acts the surface as a holomorphic automorphism. 

For the above example, $\dim |-K_X|$ is zero and $\dim |-2K_X|$ is one. 
Actually, we have 
\begin{align*}
|-2K_X|: \qquad & \begin{array}{rl} 0=&k f_0(x,y)- f_1(x,y)
\\ =&
k x^2y^2- 
 \Big(ix(x+1)^2 -i (x+i)(x^2-1)y\\& + b(x-1)^2y^2\Big)
  \Big(-ix(y-1) + y(by-1)\Big)\\\end{array}.
\end{align*}
By $\varphi$, the parameter
$$k=\frac{f_1(x,y)}{f_0(x,y)}$$
is mapped to $-k$. 
So, $$k^2=\left(\frac{f_1(x,y)}{f_0(x,y)}\right)^2$$ is the conserved quantity
and $\varphi$ belongs to Case ii-2.

We found this example by observing the following facts:
\begin{itemize}
\item  Let $X$ be a generalized Halphen surface of multiplicative type, then
$\exp(\chi(-K_X))$ is the parameter $q$ of the corresponding $q$-discrete 
Painlev\'e equation. From Theorem~\ref{thm1},
if $q$ is a primitive $m$-th root of unity, then $\dim |-k K_X|$ is
$0$ for $k=1,2,\dots,m-1$ and $1$ for $k=m$. \\
\item Let $\psi$ be an automorphism of the surface. 
If there exists another automorphism $\sigma$ of the surface such that 
$\sigma$ acts the base space of $|-m K_X|$ nontrivially, then 
$\varphi=\sigma \circ \psi$ belongs to the case ii-2 unless it is finite order.
\end{itemize}

First, we consider the family of generalized Halphen surfaces of type $A_5^{(1)}$. 
Those surfaces are obtained by blow-ups from 
$\P^1\times \P^1$ at 8 points:
\begin{eqnarray*}
&P_1:(x,y)=(b_1,0),
&P_2:(x,y)=(0,1/b_2)\\
&P_3:(x,y)=(1,\infty),
&P_4:(x,y)=(\infty,1)\\
&P_5:(x,y)=(0,0),
&P_6:(x,y/x)=(0,c)\\
&P_7:(x,y)=(\infty,\infty),
&P_8:(x,x/y)=(\infty, 1/(c b_0)).
\end{eqnarray*}
The anti-canonical divisor $xy=0$ is decomposed by
\begin{align*}
&H_x-E_2-E_5,\ E_5-E_6,\ H_y-E_1-E_5, \\
&H_x-E_4-E_7,\ E_7-E_8,\ H_y-E_3-E_7,
\end{align*}
and their orthogonal complement is generated by
\begin{align*}
\alpha_0&=H_x+H_y-E_5-E_6-E_7-E_8\\ 
\alpha_1&=H_x-E_1-E_3\\ 
\alpha_2&=H_y-E_2-E_4\\
\beta_0&=H_x+H_y-E_1-E_2-E_7-E_8\\
(\beta_1&=H_x+H_y-E_3-E_4-E_5-E_6).
\end{align*}
The period map $\chi:Q\to \C$ for the same $\omega$ with \eqref{omega} is computed as
\begin{align*}
\chi(\alpha_0)=-\log b_0,\ 
\chi(\alpha_1)=-\log b_1,\ 
\chi(\alpha_2)=-\log b_2,\\ 
\chi(\beta_0)=-\log(-cb_0b_1b_2),\
\chi(\beta_1)=\log(-c),
\end{align*}
and therefore $\chi(-K_X)=-\log(b_0b_1b_2)$.
We set $q=(b_0b_1b_2)^{-1}$.

\begin{figure}[ht]
\begin{center}
\includegraphics[width=6cm]{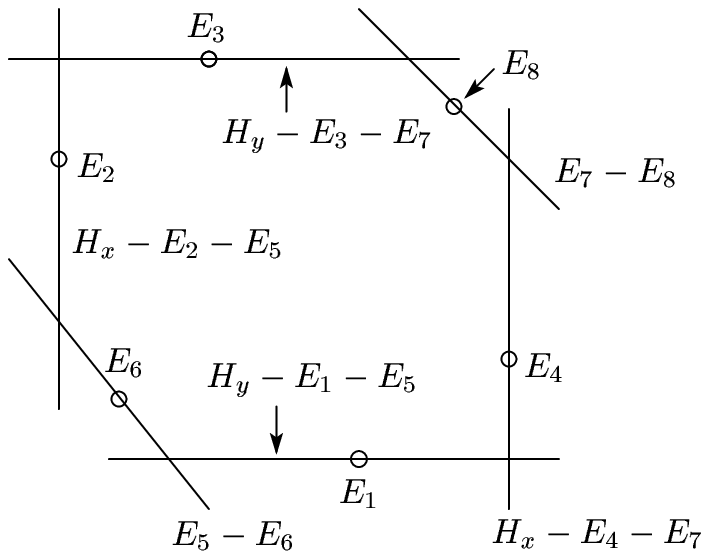}
\mbox{\raisebox{10mm}{\includegraphics[width=4cm]{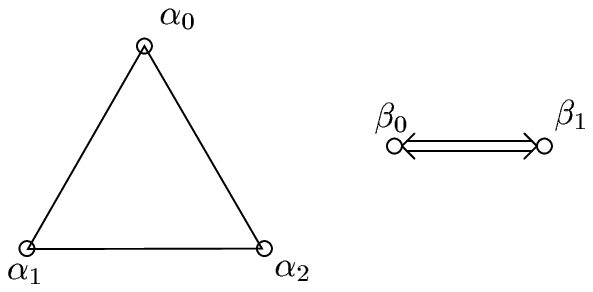}}}
\caption{Singular fiber and orthogonal complement.}
\end{center}
\end{figure}

The following actions generate the group of automorphisms 
of the family of surfaces, whose type is $A_2^{(1)}+ A_1^{(1)}$:
\begin{align*}
& (x,y;~b_0,b_1,b_2,c) \mbox{ is mapped to }\\
w_{\alpha_1}:&
\left(\frac{x}{b_1},\frac{y(x-1)}{x-b_1};~b_0b_1,
\frac{1}{b_1},b_1b_2,c\right)\\
w_{\alpha_2}:& 
\left(\frac{b_2x(y-1)}{b_2y-1},b_2 y;~b_0b_2,
b_1b_2,\frac{1}{b_2},c\right)\\
\pi:&
\left(y,x;~\frac{1}{b_0},
\frac{1}{b_1},\frac{1}{b_2},\frac{1}{c}\right)\\
\rho:&  
\left(\frac{1}{y},\frac{cx}{y};~b_1,
b_2,b_0,c\right)\\
w_{\beta_1}:&  
\left(-\frac{cx(xy-x-y)}{xy+cx-y},
-\frac{y(xy-x-y)}{(cxy-cx+y)}
;~b_0,b_1,b_2,\frac{1}{c}\right)\\
\sigma:&
\left(\frac{b_1}{x},\frac{1}{b_2 y};~b_0,
b_1,b_2,\frac{1}{b_0b_1b_2c}\right),
\end{align*}
$w_{\alpha_0}=\rho^{-1}\circ  w_{\alpha_2}\circ   \rho$ and 
$w_{\beta_0}=\sigma\circ  w_{\beta_1}\circ  \sigma$.

Here, $w_{\alpha_i}$ acts
as the elementary reflection of the affine Weyl group of type $A_2^{(1)}$ on $\bigoplus_i \Z\alpha_i$ and
trivially on $\bigoplus_j \Z\beta_j$. Similarly, 
$w_{\beta_i}$ acts trivially on $\bigoplus_i \Z\alpha_i$ and  
as the elementary reflection of the affine Weyl group of type $A_1^{(1)}$ on 
$\bigoplus_j \Z\beta_j$.  
The generators $(\alpha_0, \alpha_1, \alpha_2, \beta_0, \beta_1)$ are mapped 
by $\pi$, $\rho$, $\sigma$ to  
\begin{align*}
\pi:& (\alpha_0, \alpha_2, \alpha_1, \beta_0, \beta_1)\\
\rho:& (\alpha_2, \alpha_0, \alpha_1, \beta_0, \beta_1)\\
\sigma:& (\alpha_0, \alpha_1, \alpha_2, \beta_1, \beta_0).
\end{align*}

If $q=(b_0b_1b_2)^{-1}=-1$, then $\chi(-K_x)=-\log(-1)=-\pi i \mod 2\pi i \Z$,
and $|-2K_X|$, i.e.
the set of curves of degree $(4,4)$ passing through the blow-up points with multiplicity 2, is given by
\begin{align*}
&k_0 x^2y^2 +k_1\Big(c^2x^4(y-1)^2 + 2b_1cx^2(cxy-cx + y + b_2y^2(xy-x-y)) +\\
&  b_1^2(c^2x^2 + 2cxy(b_2y-1) + (y + b_2y^2(x-1))^2)
\Big)=0.
\end{align*}
Moreover, if $c=i$, then $\sigma$ acts identically on the parameter space 
and maps $k_1/k_0$ to $-k_1/k_0$.

Let $\psi=(w_{\alpha_1}\circ w_{\alpha_2}\circ \rho)^2$, where
$w_{\alpha_1}\circ w_{\alpha_2}\circ  \rho$ is the original $qP_{{\rm III}}$ equation
\begin{align*}
&(x,y;~b_0,b_1,b_2,c)\\
&\mapsto
\left(
\frac{cx-y}{b_2y(b_0cx-y)},
\frac{cx(b_0(cx-y)-y(b_0cx-y))}{y((cx-y)-b_2y(b_0cx-y))};~
\frac{b_0}{q},b_1q,b_2,c
\right)
\end{align*}
and acts on the root lattice
as
\begin{align*}
&(\alpha_0,\alpha_1,\alpha_2,\beta_0,\beta_1)\\
&\mapsto
(\alpha_0-\delta,\alpha_1+\delta,\alpha_2,\beta_0,\beta_1)\\
&(\delta=\alpha_0+\alpha_1+\alpha_2=\beta_0+\beta_1).
\end{align*}
Then the mapping $\psi$ acts trivially on the parameter space. Since it is very intricate mapping, we restrict the parameters to $b_0=1/b$, $b_1=-1$ and $b_2=b$, then we have $\varphi=\sigma\circ  \psi$ as \eqref{exii-2}, which acts 
on the root lattice as
\begin{align*}
&(\alpha_0,\alpha_1,\alpha_2,\beta_0,\beta_1)\\
&\mapsto
(\alpha_0-2\delta,\alpha_1+2\delta,\alpha_2,\beta_1,\beta_0).
\end{align*}

As a conclusion, the mapping $\sigma\circ w_{\alpha_1}\circ w_{\alpha_2}\circ \rho \circ w_{\alpha_1}\circ w_{\alpha_2}\circ \rho$ with the full parameter $b_1,b_2,b_3,c$
is one of $qP(A_5^{(1)}$) equation and gives the mapping by specializing of the parameters.

\section{Conclusions}

In this paper we tried to extend the algebraic-geometric approach (given in \cite{Sakai01}, \cite{DF01}) to second order mappings preserving elliptic fibrations. The case of invariants of higher degree occurs when corresponding surface
is a Halphen surface of higher index or when fibers are exchanged each other. 
We gave a classification of such mappings and also a theorem which characterize their space of initial conditions as Halphen pencils of higher index. Finally three examples showed explicit realizations of the above mentioned results. Existence of examples of Case ii-$m$ for $m\geq 3$ is a future problem.  

A finer classification may be done by the types of singular fibers and the automorphism of surfaces. Indeed, the symmetries of generalized Halphen surfaces have a close relationship with the Mordell-Weil lattice of rational elliptic surfaces. However, there are too many types of surfaces and we gave a coarse but useful classification in this paper.

\subsection*{Acknowledgement}

T.~T. is supported by the Japan Society for the
Promotion of Science, Grand-in-Aid for Young Scientists (B).
A. ~S. Carstea is supported by the project PN-II-ID-PCE-2011-3-0137. 
Also the authors
want to thanks the referees for observations about text and missed
references.

\end{document}